\documentclass[12pt,oneside,english]{amsart}
\usepackage{amssymb,amsmath,latexsym,amsthm}

\usepackage[a4paper,top=3cm,bottom=3cm,left=3cm,right=3cm,marginparwidth=1.75cm]{geometry}
\usepackage{graphicx}

\usepackage{xcolor}

\newtheorem{theorem}{Theorem}[section]
\newtheorem{corollary}[theorem]{Corollary}
\newtheorem{lemma}[theorem]{Lemma}
\newtheorem{prop}[theorem]{Proposition}

\theoremstyle{definition}

\newcommand\blfootnote[1]{%
  \begingroup
  \renewcommand\thefootnote{}\footnote{#1}%
  \addtocounter{footnote}{-1}%
  \endgroup
}
\newcommand{\PSH}{{\rm PSH}}

\newcommand{\capa}{{\rm Cap}}

\numberwithin{equation}{section}

 \usepackage{hyperref}
\hypersetup{  
    pdftitle={Full mass classes},    
    pdfauthor={Mohammed Salouf},     
    colorlinks=true,       
   linkcolor=black,          
    citecolor=black,        
    filecolor=black,      
    urlcolor=black}



\begin{document}

 \title[Degenerate complex Monge-Amp\`ere type equations]{Degenerate complex Monge-Amp\`ere  equations on some compact Hermitian manifolds}  
\author{Mohammed Salouf}

\address{ 
Department of Mathematics, 
Faculty of Sciences El Jadida, 
Chouaib Doukkali University.
24000 El Jadida. Morocco.} 
\email{salouf.m@ucd.ac.ma}

\date{\today}
   
\maketitle

\begin{abstract}
Let $X$ be a compact complex manifold which admits a hermitian metric satisfying a curvature condition introduced by Guan-Li. Given a semipositive form $\theta$ with positive volume,  we define the Monge-Amp\`ere operator for unbounded $\theta$-psh functions and prove that it is continuous with respect to convergence in capacity. We then develop pluripotential tools to study degenerate complex Monge-Amp\`ere equations in this context, extending recent results of Tosatti-Weinkove, Kolodziej-Nguyen, Guedj-Lu and many others who treat bounded solutions. 
\end{abstract}

\blfootnote{The author is supported by the CNRST within the framework of the Excellence Research Grants Program under grant number 18 UCD2022.}

\tableofcontents

\section{Introduction}

The complex Monge-Amp\`ere equation is a natural multidimensional generalizations of the Laplace equation. Starting from the seventies, Bedford and Taylor \cite{BT76} have laid down the foundation of pluripotential theory to study the Dirichlet problem in euclidean domains. The equation has also marked its appearance in the construction of canonical metrics on K\"ahler manifolds as shown in the seminal works of Calabi \cite{Cal57}, Aubin \cite{Aub78} and Yau \cite{Yau78}. As evidenced in the minimal model program in birational geometry, it is of major importance to study canonical metrics on singular varieties, see \cite{BCHM}, \cite{EGZ09}. Constructing these objects requires one to consider degenerate complex Monge-Amp\`ere equations and work with unbounded quasi-plurisubharmonic functions. Inspired by the pioneering works of Cegrell \cite{Ceg98,Ceg04}, Guedj and Zeriahi \cite{GZ07} have defined and studied non-pluripolar Monge-Amp\`ere measures of unbounded potentials. As seen in many important contributions from  \cite{BB17}, \cite{BBEGZ19}, \cite{BDL17,BDL20}, \cite{DR17}, \cite{ChCh21a,ChCh21b}, and many others, geometric pluripotential theory has occupied a central place in K\"ahler geometry during the last two decades.

However, carrying a K\"ahler metric is a restrictive condition while hermitian metrics can be easily constructed on any compact complex manifold by means of a partition of unity. Complex Monge-Amp\`ere equations on these manifolds were first appeared in the works of 
Cherrier \cite{Cher87} and Hanani \cite{Han96} who generalized Yau's theorem  under some curvature condition. After attempts from Guan and Li \cite{GL10}, Tosatti and Weinkove \cite{TW10} solved the equation for smooth solutions in full generality. Bounded weak solutions have been studied extensively over the last few years by many authors  \cite{DK12, KN15Phong, Ngu16, GL22, GL23Crelle, KN23CVPDE}. 

Compared to the K\"ahler case, the main difficulty in defining the Monge-Amp\`ere measure for unbounded potentials lies in the variation of the Monge-Amp\`ere volume.    
In the next part of the paper we will restrict ourselves to compact complex manifolds admitting a hermitian metric $\omega$ such that 
 \[
 dd^c \omega=0 \; \text{and}\; d\omega \wedge d^c \omega=0. 
 \]
 This curvature property was considered by Guan and Li in \cite{GL10} and we call it the Guan-Li condition. As shown by Chiose \cite{Chi16}, it is equivalent to the validity of the comparison principle. 
Note that any compact complex manifold admits a Gauduchon metric \cite{Gau77}, i.e. metrics $\omega$ such that $dd^c (\omega^{n-1})=0$. In particular, in dimension two Gauduchon metrics satisfy the Guan-Li condition by obvious dimensional reason.

Given a smooth real $(1,1)$-form $\theta$ with positive volume $\int_X \theta^n>0$, we follow Guedj-Zeriahi to define the non-pluripolar Monge-Amp\`ere measure $(\theta+dd^c u)^n$ for any $\theta$-psh function $u$. The class $\mathcal{E}(X,\theta)$ consists of $\theta$-psh functions with full Monge-Amp\`ere mass. %Our next result is a counterpart of Theorem \ref{thm: cv cap bounded} to unbounded functions in $\mathcal{E}(X,\theta)$. It is stated as follows. 

%In constructing weak solutions one often needs to use convergence in capacity of the potentials, which guarantees the convergence of the Monge-Amp\`ere measures. Our first main result in this direction is the following

A famous example due to Cegrell  \cite[Example 3.25]{GZ17} shows that the Monge-Amp\`ere operator is not continuous with respect to the $L^1$-topology. Xing \cite[Theorem 1]{Xin09} showed that it is continuous with respect to convergence in capacity of uniformly bounded psh functions (for the definition of convergence in capacity, we refer to Section \ref{sect: cont of MA}).  This is the weakest type of convergence that ensures convergence of Monge-Amp\`ere measures.  In our result below, which fits into this context, we prove that convergence in $L^1$ together with  domination by a fixed positive non-pluripolar measure imply convergence in capacity, hence convergence of the Monge-Amp\`ere measures.  

\begin{theorem}\label{main: cv cap}
  Assume $u_j,u \in \mathcal{E}(X,\theta)$ and $u_j\to u$ in $L^1(X,\omega^n)$. 
  \begin{enumerate}
      \item If $(\theta + dd^c u_j)^n \leq \mu$, where $\mu$ is a positive non-pluripolar measure, then $u_j$ converges to $u$ in capacity. 
      \item If $u_j\to u$ in capacity then $(\theta+dd^c u_j)^n$ weakly converges to $(\theta+dd^c u)^n$. 
  \end{enumerate}
\end{theorem}

When  $\theta$ is K\"ahler, the first part of Theorem \ref{main: cv cap} returns to \cite{DV22} (we refer to \cite{Xin09} and \cite{DH12} for some related results).  The second part was proved in \cite{Xin09} and generalized to the setting of relative pluripotential theory in \cite{DDL23}. 

%It generalizes previous results obtained in \cite{Xin09} and \cite{DH12}. 
%Part (2) is well-known in the K\"ahler case 

In our proof of (1), we use the rooftop envelopes as in \cite{Dar15,Dar17} to produce an increasing sequence of $\theta$-psh functions lying below $u_j$ and converging to $u$. Thus, we actually prove that the sequence converges quasi monotone in the terminology of \cite{GT23}. In order to do so, we need the condition that $(\theta+dd^c u_j)^n\leq \mu$ to ensure that the quasi-plurisubharmonic envelopes are not identically $-\infty$.  
Using the same idea, we also obtain in Theorem \ref{thm: cv cap bounded} (where the Guan-Li metric is not needed) an analogue of the first statement in Theorem \ref{main: cv cap} for sequences of uniformly bounded $\theta$-psh functions, positively answering a question of Ko{\l}odziej and Nguyen \cite{KN22}. %Note that when $\theta$ is closed and $X$ is K\"ahler, the condition in Theorem \ref{thmB} implies that $u_j$ converges to $u$ in energy which implies convergence in capacity (see \cite{BBEGZ19}, \cite{GLZ19}).  

%Having the last theorem, it is natural to study the continuity of the operator $(\theta+ dd^c.)^n$ with respect to sequences in $\mathcal{E}(X,\theta)$ that converge in capacity. This turns out to be the case when $\theta$ is a K\"ahler form as shown by Xing \cite{Xin09}. The following result is a generalization of \cite[Theorem 1]{Xin09} to our settings. 
%\begin{theorem}\label{thmC}
 %   Let $u_j$, $u \in \mathcal{E}(X,\theta)$. If $u_j \rightarrow u$ in capacity, then $(\theta+dd^cu_j)^n$ converges to $(\theta +dd^c u)^n$ in the weak sense of measures. 
%\end{theorem}

We next follow ideas of Cegrell \cite{Ceg98} and Ko{\l}odziej-Nguyen \cite{KN22} to solve degenerate complex Monge-Ampère equations with right-hand side being a positive non-pluripolar meaure $\mu$. We first write $\mu = f(\theta + dd^c \phi)^n$, for a bounded $\theta$-psh function $\phi$. We approximate $\phi$ and $f$ to obtain a sequence of regular measures $\mu_j$, and we solve the corresponding Monge-Ampère equations. We finally use Theorem \ref{main: cv cap} to pass to the limit, obtaining the desired solution. 
 %Having the last theorem, a natural question to arises is whether there are solutions to the equation \eqref{main-eqt}. This is the objective of the following theorem. 
 \begin{theorem}\label{main: solutions}
     Assume $\mu$ is a positive Radon measure vanishing on pluripolar sets. Then,
   \begin{enumerate}
       \item  there is a unique constant $c>0$ and a function $u \in \mathcal{E}(X,\theta)$ such that 
         $$ (\theta + dd^c u)^n = c d\mu; $$
         \item  if $\lambda>0$, then there is a uniquely determined $u \in \mathcal{E}(X,\theta)$ such that 
         $$ (\theta + dd^c u)^n = e^{\lambda u} d\mu. $$
   \end{enumerate}
 \end{theorem}

% We start by writing $\mu = f(\theta + dd^c \phi)^n$, for a bounded $\theta$-psh function $\phi$ and $f \in L^1((\theta + dd^c \phi)^n)$.
 %We first prove the theorem when the function $f$ is bounded. Then, we approximate the measure $\mu$ by the sequence $\mu_j = \min(f,j) (\theta + dd^c \phi)^n$, $j \in \mathbb{N}^*$.
%Consider   $u_j \in \mathcal{E}^1(X,\theta)$ such that
 %$$ (\theta + dd^c u_j)^n = e^{\lambda u_j} d\mu_j, \; \; \forall j \in \mathbb{N}^*. $$
%We prove that the sequence $(u_j)$ has a subsequence converging in $L^1(X)$ to a $\theta$-psh function $u$ that solves \eqref{main-eqt} (Theorem \ref{thm 4.3}).
%The proof is based on the following convergence theorem  that generalizes \cite[Proposition 2.8]{GZ07} and \cite[Theorem 1.1]{KN23b}. 
%\begin{thmC}
%    Let $\chi \in \mathcal{W}^-$, and let $(u_j) \subset \mathcal{E}(X,\omega)$ be a sequence of non-positive functions converging to a $\omega$-psh function  $u$  in $L^1(X)$. Assume that
%$$ \sup_j \int_X -\chi(u_j)(\omega + dd^c u_j)^n <+\infty,$$
%and
%$$ \lim_{j\rightarrow +\infty} \int_X |\max(u_j,-k) - \max(u,-k)|(\omega + dd^c u_j)^n = 0, $$
%    for every $k \in \mathbb{N}$. Then, $u \in \mathcal{E}(X,\omega)$ and  $((\omega + dd^c u_j)^n)_j$  converges weakly  to  $(\omega + dd^c u)^n$.  
%\end{thmC}
  %In the last theorem, $\mathcal{W}^-$ denotes the set of convex increasing functions $\chi: \mathbb{R}_- \rightarrow \mathbb{R}_-$  satisfying $\chi(-\infty)=-\infty$.

  When $\lambda=0$, the uniqueness of solution is still an open problem even in the case $\mu = f\omega^n$, $f\in L^p(X)$ and $p>1$; we refer to \cite{KN19}  for more information. When $\lambda>0$, and $\mu=fdV$ with $f\in L^p$, $p>1$,  the uniqueness of solution was proved in  \cite{Ngu16} using the modified comparison principle. The latter breaks down when dealing with unbounded solutions. 
  
  %  as in K\"ahler setting  \cite{BBGZ13}. On non-K\"ahler manifolds, a modified comparison principle was established  in \cite[Theorem 0.2]{KN15Phong} and used to prove the uniqueness of bounded solutions in the case  $\mu = f \omega^n$, $f \in L^p(\omega^n)$ and $p>1$ \cite[Lemma 3.6]{Ng16AIM}. 
   To prove uniqueness in our case, we adapt a new method using quasi-psh envelopes as proposed in \cite{LN22} and \cite[Section 2.3]{GL22}. It is based on the following domination principle. 
  \begin{theorem}\label{main: domination principle}
      Fix $c \in [0,1)$. If  $u$, $v \in \mathcal{E}(X,\theta)$ are such that 
$$ (\theta + dd^c u)^n \leq c(\theta + dd^c v)^n $$
on $\{u<v\}$, then $u \geq v$.
  \end{theorem}
 
%To prove the last theorem, we adapt the proof of \cite[Proposition 2.8]{GL22}. We consider the functions, 
%$$ \varphi_j = \sup \{ h \in \PSH(X,\omega) : h \leq j u - (j-1)v \}, \; \; j \in \mathbb{N}^*. $$
%The main problem is to prove $(\varphi_j) \subset \mathcal{E}(X,\omega)$ and the existence of $\varphi \in \PSH(X,\omega)$ such that  $ \varphi_j \rightarrow \varphi$ in $L^1(X)$. 

  The paper is organized as follows. In Section \ref{sect: NP measures}, we define the non-pluripolar Monge-Ampère measures of $\theta$-plurisubharmonic functions, and we study basic properties of the Monge-Amp\`ere operator in the full mass class $\mathcal{E}(X,\theta)$. In Section \ref{sect: cont of MA}, we study the convergence in capacity of functions in $\mathcal{E}(X,\theta)$ and prove Theorem \ref{main: cv cap}.  Section \ref{sect: solutions to MA-E} is devoted to solving degenerate Monge-Ampère equations. 
  \medskip 

{\bf Acknowledgment.}
 The author would like to express his deepest gratitude to his supervisors Omar Alehyane  and Chinh H. Lu  for many valuable discussions and for their useful suggestions and comments. He also thanks Prof. Duc-Viet Vu  for pointing out the paper \cite{DV22}, where the statement (1) in Theorem \ref{main: cv cap} is proved in the K\"ahler case. 
 The paper was written while the author was visiting Université d'Angers. He is grateful to the university's members for their warm welcoming. The stay was financed by the Doctoral Mobility Support Program-Embassy of France in Morocco.
\section{Non-pluripolar Monge-Amp\`ere measures} \label{sect: NP measures}
\subsection{Definitions}
Let $X$ be a compact complex manifold of dimension $n$, and let $\omega$ be a hermitian form on $X$ satisfying the Guan-Li condition 
\[
dd^c \omega =0 \; \text{and} \; d\omega \wedge d^c \omega=0. 
\]
Let $\theta$ be a smooth real $(1,1)$-form which is semipositive and big. Up to rescaling $\omega$, there is no loss of generality in assuming $\theta \leq \omega$.

We say that a function $h : X \longrightarrow \mathbb{R} \cup \{-\infty\}$ is quasi-plurisubharmonic (quasi-psh for short) if it is locally given by the sum of a smooth and a psh function. In the sequel, we will denote by $\PSH(X,\theta)$ the set of quasi-psh functions $u \in L^1(X)$ satisfying  $\theta + dd^c u\geq 0$ in the sense of currents. Functions in $\PSH(X,\theta)$ are called $\theta$-plurisubharmonic; we will use the notation $\theta$-psh for short. If $u \in \PSH(X,\theta)$, then we set $\theta_u = \theta + dd^c u$. 

In \cite{GL22}, the authors defined the following quantities: 
$$ v_+(\theta) = \sup \left\{ \int_X (\theta + dd^c u)^n : \; u \in \PSH(X,\theta) \cap L^\infty(X) \right\}, $$
and 
$$ v_-(\theta) = \inf \left\{ \int_X (\theta + dd^c u)^n : \; u \in \PSH(X,\theta) \cap L^\infty(X)  \right\}. $$
It is of great interest to know whether $v_+(\theta)<+\infty$ and/or $v_-(\theta)>0$.

 Since, in our settings, $\omega$ satisfies the Guan-Li condition, we get by Stokes' theorem that $v_-(\omega)=v_+(\omega)=\int_X\omega^n$. Since $\theta \leq \omega$, we obtain $v_+(\theta) \leq v_+(\omega)<\infty$. We prove in Corollary \ref{cor: v- positive} that we also have $v_-(\theta)>0$.

The complex Monge-Ampère operator $(\theta + dd^c .)^n$ is well-defined on bounded $\theta$-psh functions by Bedford-Taylor's construction \cite{BT76}. We are going to extend this definition to unbounded $\theta$-psh functions, by following the work of Guedj-Zeriahi \cite{GZ07}. 

Let $u \in \PSH(X,\theta)$, and  set  $u_t =\max(u,-t)$, for   every $t \in \mathbb{R}^+$.
 It follows from the plurifine property of the Monge-Ampère operator on bounded functions \cite{BT76} (see also \cite[Theorem 5.3]{Sal23}) that
$$ \textit{1}_{\{u_{t}>-s\}}(\theta + dd^c u_{t})^n = \textit{1}_{\{u_{t}>-s\}}(\theta + dd^c \max(u_{t},-s))^n,  $$
for every $t$, $s\in \mathbb{R}_+$. If $t\geq s$, then 
\begin{align*}
   \textit{1}_{\{u>-t\}}(\theta + dd^c u_{t})^n &\geq 
   \textit{1}_{\{u>-s\}}(\theta + dd^c u_{t})^n \\
   &=  \textit{1}_{\{u_{t}>-s\}}(\theta + dd^c u_{t})^n \\
   &= \textit{1}_{\{u>-s\}}(\theta + dd^c u_{s})^n. 
\end{align*}
We conclude that the sequence of general term
$$ \textit{1}_{\{u> -t\}} (\theta + dd^c \max(u,-t))^n, \; \; t \geq 0, $$
is non-decreasing. 
Since the total mass of these measures are uniformly bounded from above by $v_+(\theta)$, the non-pluripolar Monge-Ampère measure $(\theta+ dd^cu)^n$ is well-defined by the strong limit
$$ \theta_u^n = \lim_{t\rightarrow +\infty }  \textit{1}_{\{u> -t\}} (\theta + dd^c \max(u,-t))^n. $$

%We notice that, if $u^1,...,u^n \in \mathcal{E}(X,\theta)$, then the mixed Monge-Ampère measure 
%$$\theta_{u^1} \wedge ... \wedge \theta_{u^n} := \lim_{t \rightarrow +\infty} \textit{1}_{\cap_j \{u^j >-t\}} \theta_{\max(u^1,-t)} \wedge ... \wedge \theta_{\max(u^n,-t)} $$ 
%defines a positive Radon measure on $X$. 

The following result, known as the maximum principle, is a classical tool in the theory of Monge-Ampère equations. We will use it on several occasions.  
\begin{theorem}\label{thm: plurifine property} 
    If $u$, $v\in \PSH(X,\theta)$, then  
    $$  \textit{1}_{\{u>v\}} (\theta + dd^cu)^n = \textit{1}_{\{u>v\}} (\theta + dd^c \max(u,v))^n.   $$
\end{theorem}
\begin{proof}
 For each $j\in \mathbb{N}$, we set $u_j=\max(u,-j)$ and $v_j=\max(v,-j)$. We have by the plurifine property of the operator $(\theta+dd^c.)^n$ on bounded functions  that 
  $$  \textit{1}_{\{u_j>v_j\}} \theta_{u_j}^n = \textit{1}_{\{u_j>v_j\}} \theta_{\max(u_j,v_j)}^n.   $$
  Hence, 
    $$ \textit{1}_{\{u>v\}\cap \{v>-j\}} \theta_{u_j}^n = \textit{1}_{\{u>v\}\cap \{v>-j\}} \theta_{\max(u_j,v_j)}^n.   $$
   Again, we get by the plurifine property that 
  $$  \textit{1}_{\{u>-j\}} \textit{1}_{\{u>-k\}}   \theta_{u_k}^n =  \textit{1}_{\{u>-j\}} \theta_{u_k}^n = \textit{1}_{\{u>-j\}} \theta_{u_j}^n.  $$
   for every $k\geq j$. 
   Letting $k \rightarrow +\infty$ yields 
  $$   \textit{1}_{\{u>-j\}} \theta_{u_j}^n =  \textit{1}_{\{u>-j\}} \theta_{u}^n, $$
   by the definition of the operator $(\theta + dd^c.)^n$. We get similarly that 
   $$ \textit{1}_{\{u>v\}\cap \{v>-j\}} \theta_{\max(u_j,v_j)}^n = \textit{1}_{\{u>v\}\cap \{v>-j\}} \theta_{\max(u,v)}^n.  $$
   That implies 
   $$ \textit{1}_{\{u>v\}\cap \{v>-j\}} \theta_{u}^n = \textit{1}_{\{u>v\}\cap \{v>-j\}} \theta_{\max(u,v)}^n.   $$
    The result thus follows after letting $j\rightarrow +\infty$, since the Monge-Ampère measure $(\theta + dd^c h)^n$ does not charge pluripolar sets for every $h\in\PSH(X,\theta)$.     
\end{proof}
The following result is a consequence of the last theorem.
\begin{corollary}\label{cor: MA inequality}
    Let $u$, $v \in \PSH(X,\theta)$. We have, 
      $$ \theta_{\max(u,v)}^n \geq \textit{1}_{\{u>v\}} \theta_{u}^n + \textit{1}_{\{u\leq v\}} \theta_{v}^n.$$
  If moreover $u \leq v$, then 
    $$ \textit{1}_{\{u = v\}} \theta_v^n \geq \textit{1}_{\{u = v\}} \theta_u^n. $$
\end{corollary}
\begin{proof}
The proof of the first inequality is exactly the same as that of \cite[Lemma 2.9]{DDL23}. The second one follows easily from the first. 
\end{proof}

\subsection{Quasi-plurisubharmonic envelopes}
If $h : X \rightarrow \mathbb{R}_-\cup\{\pm \infty\}$ is a measurable function, then  the $\theta$-psh envelope of $h$ is defined by 
$$ P_{\theta}(h) = \left(\sup \{ \varphi \in \PSH(X,\theta) : \varphi \leq h\}\right)^*,  $$
with the convention $\sup\emptyset = -\infty$. When $h=\min(u,v)$, we use the notation $P_\theta(u,v)=P_\theta(\min(u,v))$.

During the last decades, the usage of these envelopes has  become a strong method for studying degenerate Monge-Ampère equations for both local and global cases. We refer to \cite{BT76, Ber19, CZ19, GLZ19, GL22, DDL23} to name a few. 

The following theorem is in the heart of the theory of quasi-psh envelopes.  
%We refer to \cite[Theorem 2.3]{GL22} for a proof of this result in the case when $h$ is bounded from below. 

\begin{theorem}\label{thm: env}
    Assume $h$ is quasi-continuous on $X$ and $P_\theta(h) \in \PSH(X,\theta)$. Then, $P_\theta(h) \leq h$ outside a pluripolar set, and we have 
    $$ \int_{\{P_\theta(h)<h\}} (\theta + dd^c P_\theta(h))^n = 0. $$ 
\end{theorem}
\begin{proof}    
    For each $j \in \mathbb{N}$, we set $h_j=\max(h,-j)$. It follows from \cite[Theorem 2.3]{GL22} that $P_\theta(h_j)\leq h_j$ outside a pluripolar set. Since a countable union of pluripolar sets is pluripolar, there is a pluripolar set $P$ such that $P_\theta(h_j) \leq h_j$ on $X\setminus P$ for every $j$. That implies  $P_\theta(h) \leq h$ on $X\setminus P$.

    We move on to the proof that $(\theta+dd^cP(h))^n$ is carried by the set $\{P_\theta(h)=h\}$.
    According to \cite[Theorem 2.3]{GL22}, we have 
    $$ \int_{\{P_\theta(h_j)<h_j\}} (\theta + dd^c P_\theta(h_j))^n=0, \; \; \forall j \geq 1. $$
%    Fixing $C>0$,  we get by Theorem \ref{thm: plurifine property} that 
%\begin{align*}
%    &\int_{\{P_\theta(h_j)<h_j\}\cap \{P_\theta(h) >-C\}} (\theta + dd^c \max(P_\theta(h_j),-C))^n \\
 %   &=\int_{\{P_\theta(h_j)<h_j\}\cap \{P_\theta(h) >-C\}} (\theta + dd^c \max(P_\theta(h_j))^n =0. 
%\end{align*}
Fix $k \in \mathbb{N}$. We have for every $j\geq k$
 \[
\int_{\{P_\theta(h_k)<h\}} (\theta + dd^c P_\theta(h_j))^n=0, 
\]
since in this case $\{P_\theta(h_k)<h\} \subset \{P_\theta(h_j)<h_j\}$.
Fixing $C>0$, we get by Theorem \ref{thm: plurifine property} that 
 \[
\int_{\{P_\theta(h_k)<h\}\cap \{P_\theta(h) >-C\}} (\theta + dd^c \max(P_\theta(h_j),-C))^n=0. 
\]
Notice that the function 
$$ f_k = \left(\max(h,P_\theta(h_k)) - P_\theta(h_k)\right) \times \left( \max(P_\theta(h),-C) + C \right)  $$
is a positive bounded quasi-continuous function that satisfies 
$$ \int_X f_k (\theta + dd^c \max(P_\theta(h_j),-C))^n = 0, \; \forall j \geq k. $$
The sequence $(\max(P_\theta(h_j),-C))_j$ is uniformly bounded and decreases to $\max(P_\theta(h),-C)$. Therefore, we get  
the Bedford-Taylor convergence theorem that  
$$ (\theta + dd^c\max(P_\theta(h_j),-C))^n \rightarrow (\theta + dd^c \max(P_\theta(h),-C))^n \; \; \textit{weakly.} $$
It thus follows from \cite[Lemma 2.4]{DDL23} that 
$$ \int_X f_k (\theta + dd^c P(h))^n =0, $$
which is equivalent to 
$$ \int_{\{P_\theta(h_k)<h\}\cap \{P_\theta(h) >-C\}} (\theta + dd^c P_\theta(h))^n = 0. $$
Since $P_\theta(h_k) \searrow P_\theta(h)$, letting $k\rightarrow +\infty$ yields
$$ \int_{\{P_\theta(h)<h\}\cap \{P_\theta(h) >-C\}} (\theta + dd^c P_\theta(h))^n=0, $$
by the monotone convergence theorem.
The result thus follows by letting $C\rightarrow +\infty$, since $(\theta+dd^c P(h))^n$ does not charge pluripolar sets. 
\end{proof}

The following result is a consequence of the last theorem. 
\begin{corollary}\label{cor: v- positive}
    There is $A\geq 1$ such that $v_-(\omega) \leq Av_-(\theta)$. In particular,  $v_-(\theta)>0$.
\end{corollary}
%Before giving the proof of the last proposition, we recall the difinition of the global extremal function.
If $D\subset X$ is a Borel set, then  
    $$ V_{D,\omega} = \sup \{h \in \PSH(X,\omega)\; : \; h \leq 0 \; \text{on } D \}, $$
    is called the global extremal function of $D$ with respect to $\omega$. 
    By \cite[Theorem 9.17]{GZ17}, $V_{D,\omega}=+\infty$ if and only if $D \subset \{h = - \infty\}$ for certain function $h \in \PSH(X,\omega)$.
    The latter is equivalent to $D$ being pluripolar by \cite{Vu19IJM}. 
    
    Now we proceed to the proof of the corollary. 
\begin{proof}
The proof was implicitly written in \cite{GL22}. 
    Consider a $\theta$-psh function $\psi$ with analytic singularities such that $\sup_X \psi=0$ and  $\theta + dd^c \psi \geq \varepsilon \omega$, for some $\varepsilon>0$. 
    Fix $u \in \PSH(X,\theta) \cap L^\infty$ and set $v=P_{\varepsilon\omega}(u-\psi)$. The set $D=\{\psi>-1\}$ is non-pluripolar because $\psi$ is quasi-psh. It follows from \cite[Theorem 9.17]{GZ07} that $V_{D,\varepsilon\omega} <+\infty$.  Taking $C = \sup_X u +1$ implies $v \leq C$  on $D$, and hence $v \leq V_{D,\varepsilon\omega} + C <+\infty$. Since $u \leq v$, we infer $v \in \PSH(X,\varepsilon\omega) \cap L^\infty$. Moreover, we have $v+\psi \in \PSH(X,\theta)$ and $v+\psi\leq u$. Therefore, we get by Theorem \ref{thm: plurifine property} and Corollary \ref{cor: MA inequality} that
    $$ v_-(\varepsilon\omega) \leq \int_X (\varepsilon \omega + dd^c v)^n \leq \int_{\{v+\psi=u\}} (\theta + dd^c(\psi+v))^n \leq \int_X \theta_u^n. $$
    Hence $v_-(\theta)\geq v_-(\varepsilon\omega)\geq \varepsilon^n v_-(\omega)$.
\end{proof}

\subsection{The full mass class}
For $u \in \PSH(X,\theta)$, we defined the non-pluripolar Monge-Ampère measure of $u$ by the strong limit
$$ \theta_u^n = \lim_{t \rightarrow +\infty} \textit{1}_{\{u>-t\}}(\theta + dd^c \max(u,-t))^n. $$
 In order to give a natural definition for the operator $(\theta + dd^c.)^n$, it is convenient to limit ourselves to $\theta$-psh functions $u$ satisfying the condition 
$$ \lim_{t\rightarrow +\infty }   (\theta + dd^c \max(u,-t))^n(\{u \leq -t\}) = 0.  $$
Let us denote by $\mathcal{E}(X,\theta)$ the set of such functions.  If $\theta$ is closed, then $\mathcal{E}(X,\theta)$ coincides with the one given in \cite{GZ07, BEGZ10} by \cite[Lemma 1.2]{GZ07}.

We define the set $\mathcal{E}(X,\omega)$ similarly. Since $\omega$ satisfies the Guan-Li conditions, it is clear that a function $u \in \PSH(X,\omega)$ belongs to $\mathcal{E}(X,\omega)$ if, and only if, $u$ satisfies
$$ \int_X (\omega+ dd^cu)^n= \int_X \omega^n.$$ 

We have the following characterizations of the class  $\mathcal{E}(X,\theta)$. 
\begin{prop}\label{prop: carc full mass}
   Let $u \in \PSH(X,\theta)$. The following assertions are equivalent:
    \begin{enumerate}
        \item $u \in \mathcal{E}(X,\theta)$;
        \item $u \in \mathcal{E}(X,\omega)$;
        \item $P_\theta(Au) \in \PSH(X,\theta)$, for every $A\geq1$.
    \end{enumerate}
\end{prop}
\begin{proof}
    We start by proving the implication $1) \Rightarrow 3)$. Assume $u \in \mathcal{E}(X,\theta)$ and set $u_j=\max(u,-j)$, $\varphi_j=P_\theta(Au_j)$. It follows from Theorem \ref{thm: env} that $\theta_{\varphi_j}^n$ is carried by the set $\{\varphi_j = A u_j\}$. Since $\varphi_j \leq A u_j$, we get by Corollary \ref{cor: MA inequality} that
    $$ \theta_{\varphi_j}^n \leq A^n \textit{1}_{\{\varphi_j = A u_j\}} \theta_{A^{-1}\varphi_j}^n \leq A^n \textit{1}_{\{\varphi_j = A u_j\}} \theta_{u_j}^n. $$
    Fix $k \in \mathbb{N}$. If $(\varphi_j)$ decreases to $-\infty$, then there is $j_0$ such that $X = \{\varphi_j \leq -A k\}$ for every $j\geq j_0$. That implies, 
    $$ v_-(\theta) \leq \int_{\{\varphi_j \leq -Ak\}} \theta_{\varphi_j}^n \leq A^n \int_{\{u_j\leq -k\}}  \theta_{u_j}^n \leq A^n\left(\int_X  \theta_{u_j}^n  -\int_{\{u_j>-k\}}  \theta_{u_j}^n\right) . $$
   Hence for $j \geq k$, we have by Theorem \ref{thm: env} that 
    $$ \int_{\{u_j>-k\}}  \theta_{u_j}^n =  \int_{\{u>-k\}}  \theta_{u_k}^n. $$
    That implies
    $$ v_-(\theta) \leq A^n \left(\int_X  \theta_{u_j}^n  -\int_{\{u>-k\}}  \theta_{u_k}^n\right),  $$
    for every $j \geq k$. 
    Letting $j \rightarrow +\infty$ yields
     $$ v_-(\theta) \leq A^n \left(\int_X  \theta_{u}^n - \int_{\{u>-k\}}  \theta_{u_k}^n\right), $$
     because $u\in \mathcal{E}(X,\theta)$. 
     Letting $k \rightarrow +\infty$ implies $v_-(\theta)=0$, a contradiction. That proves $(\varphi_j)$ does not converge uniformly to $-\infty$. Since $\varphi_j \searrow P_\theta(Au)$, we obtain $P_\theta(Au) \in \PSH(X,\theta)$.

     We move on to the proof of $3) \Rightarrow 2)$. 
     Fix $A>1$. Using the fact that  $\theta \leq \omega$, we get that $P_\omega(Au) \geq P_\theta(Au)>-\infty$; and  consequently $P_\omega(Au) \in \PSH(X,\omega)$. Since $P_\omega(Au) \leq Au$ by 
   Theorem \ref{thm: env}, we infer  by \cite[Theorem 3.3]{DDL23} that
     $$ \int_X (\omega+ dd^cu)^n \geq \int_X \omega_{A^{-1}P_\omega(Au)}^n \geq (1-A^{-1})^n \int_X \omega^n. $$
     Letting $A \rightarrow +\infty$ yields $u \in \mathcal{E}(X,\omega)$.

      The implication $2) \Rightarrow 1)$ is trivial. 
\end{proof}
\begin{corollary}\label{cor: P(Au) in full mass class}
     If $u \in \mathcal{E}(X,\theta)$, then $P_\theta(Au) \in \mathcal{E}(X,\theta)$ for every $A\geq 1$. 
\end{corollary}
\begin{proof}
Fix $A\geq 1$ and set $v=P_\theta(Au)$. It follows from the last proposition that $v \in \PSH(X,\theta)$.  By the last proposition, it suffices to prove $P_\theta(tv) \in \PSH(X,\theta)$ for every $t\geq 1$. Since $P_\theta(Atu) \leq Atv$, we get that $tv \geq A^{-1} P_\theta(Atu) \in \PSH(X,\theta)$. Hence, $P_\theta(tv) \in \PSH(X,\theta)$ and the proof is thus completed.   
\end{proof}
\begin{prop}\label{prop : stab maxima} 
    The set $\mathcal{E}(X,\theta)$ is convex. Moreover, if $u \in \mathcal{E}(X,\theta)$ and $v \in \PSH(X,\theta)$ are such that $u \leq v$, then $v \in \mathcal{E}(X,\theta)$.
\end{prop}
\begin{proof}
According to Proposition \ref{prop: carc full mass}, it suffices to prove the result for $\theta$ replaced by $\omega$.  This follows exactly as in the proof of \cite[Proposition 1.6]{GZ07} since $\omega$ satisfies the Guan-Li condition. 
\end{proof}
\begin{corollary}
 Functions in $\mathcal{E}(X,\theta)$ have zero Lelong numbers.
\end{corollary}
\begin{proof}
    This follows by the same method as in the proof of \cite[Corollary 1.8]{GZ07} by observing that $\mathcal{E}(X,\theta) \subset \mathcal{E}(X,\omega)$ by Proposition \ref{prop: carc full mass}. 
\end{proof}
\begin{lemma}\label{lem: MA of rooftop env 1}
    Assume $u,v \in \mathcal{E}(X,\theta)$. Then $P(u,v)\in \mathcal{E}(X,\theta)$. If  $\theta_u^n \leq \mu$ and $\theta_v^n\leq \mu$ for some non-pluripolar measure $\mu$, then 
    \[
    (\theta+dd^c P(u,v))^n \leq \mu.
    \]
\end{lemma}

\begin{proof}
   Let us denote by $\varphi = P_\theta(\min(u,v))$ and $h=(u+v)/2$. The function $h$ belongs to $\mathcal{E}(X,\theta)$ by Proposition \ref{prop : stab maxima}, and we have $P_\theta(2h) \leq \varphi$ by Theorem \ref{thm: env}. Corollary \ref{cor: P(Au) in full mass class} implies $P_\theta(2h) \in \mathcal{E}(X,\theta)$, and hence we obtain $\varphi \in \mathcal{E}(X,\theta)$ by the last proposition.

   According to Theorem \ref{thm: env}, the measure $\theta_\varphi^n$ is carried by the set 
   $$ \{\varphi = \min(u,v)\} = \{\varphi = u<v \} \cup \{\varphi = v\}.  $$
   Since $\varphi \leq u$, it follows from Corollary \ref{cor: MA inequality} that
   $$ \textit{1}_{\{\varphi = u<v \}} \theta_\varphi^n \leq \textit{1}_{\{\varphi = u<v \}} \theta_u^n \leq \textit{1}_{\{\varphi = u<v \}} \mu. $$
   We get similarly that 
   $$  \textit{1}_{\{\varphi = v \}} \theta_\varphi^n \leq \textit{1}_{\{\varphi = v \}} \mu. $$
   Therefore, $\theta_{P_{\theta}(u,v)}^n \leq \mu$ as desired. 
\end{proof}

\begin{lemma}\label{lem: MA of rooftop env 2}
    Assume $u,v \in \mathcal{E}(X,\theta)$, $h\in \PSH(X,\theta)$. Then   
    \[
    \int_X |e^{P(u,v)} -e^h| \theta_{P(u,v)}^n \leq \int_X |e^u -e^h| \theta_u^n+ \int_X |e^v -e^h| \theta_v^n. 
    \]
\end{lemma}
\begin{proof}
  Set $\varphi=P(u,v)$. It follows from Theorem \ref{thm: env} that $\theta_\varphi^n$ is carried by the set $\{\varphi= \min(u,v)\}$ and that $\varphi \leq \min(u,v)$. Hence, we obtain by Corollary \ref{cor: MA inequality} that 
  \begin{align*}
        \int_X |e^\varphi -e^h| \theta_\varphi^n &\leq  \int_{\{\varphi=u\}} |e^\varphi -e^h| \theta_\varphi^n +  \int_{\{\varphi=v\}} |e^\varphi -e^h| \theta_\varphi^n  
      \\
      &\leq  \int_{\{\varphi=u\}} |e^\varphi -e^h| \theta_u^n +  \int_{\{\varphi=v\}} |e^\varphi -e^h| \theta_v^n  \\
        &\leq \int_X |e^u -e^h| \theta_u^n+ \int_X |e^v -e^h| \theta_v^n. 
  \end{align*}
\end{proof}
\subsection{The domination principle}
The following result, called the domination principle, is a consequence of Theorem \ref{thm: env}. The proof is a modification of the one of \cite[Proposition 2.8]{GL22} which handles the case of bounded $\theta$-psh functions. 
\begin{theorem}[Domination principle]\label{thm : dom prin}
Let $c \in [0,1)$, and let  $u$, $v \in \mathcal{E}(X,\theta)$. If 
$$ \theta_u^n \leq c\theta_v^n $$
on $\{u<v\}$, then $u \geq v$.
\end{theorem}
\begin{proof}
We first notice that, by Theorem \ref{thm: env}, we have  
 $$ c(\theta + dd^c \max(u,v))^n = c \theta_v^n \geq \theta_u^n $$
 on $\{u <v\}$. Hence, up to replacing $v$ by $\max(u,v)$, one can assume $u \leq v$. Our objective now is to prove that $u=v$.     
    
    For $b > 1$, we define 
    $$ u_b = P_\theta(bu - (b-1)v). $$
    We have $bu - (b-1) v \leq v$ for every $b$. Hence $u_b \in \PSH(X,\theta)$ or $u_b=-\infty$. We are going to prove $u_b \in \mathcal{E}(X,\theta)$ for every $b$. 
     The fact that $u \in \mathcal{E}(X,\theta)$ implies  $P_\theta(bu) \in \mathcal{E}(X,\theta)$ by Corollary \ref{cor: P(Au) in full mass class}.  We have 
     $$P_\theta(bu) \leq bu \leq bu -(b-1)(v - \sup_X v). $$ 
    Hence $u_b \geq P_\theta(bu) - (b-1) \sup_X v \in \mathcal{E}(X,\theta)$, and consequently $u_b \in \mathcal{E}(X,\theta)$ by Proposition \ref{prop : stab maxima}.
   Let us denote by $D:= \{u_b=bu - (b-1)v\}$. According to 
    Theorem \ref{thm: env}, the measure $\theta_{u_b}^n$ is carried by $D$ and $u_b \leq b u - (b-1)v$ outside a pluripolar set. Which implies $b^{-1}u_b + (1-b^{-1})v \leq u_b$ almost everywhere, and hence everywhere because these are quasi-psh functions. 
   It thus follows from Corollary \ref{cor: MA inequality} that  
    $$ \textit{1}_{D} (\theta + dd^c (b^{-1} u_b + (1 - b^{-1})v))^n \leq \textit{1}_D \theta_u^n. $$
    Hence, we get by the hypothesis that
    $$ \textit{1}_{D\cap \{u<v\}} b^{-n}  \theta_{u_b}^n + \textit{1}_{D\cap \{u<v\}}(1 - b^{-1})^n\theta_v^n \leq \textit{1}_{D\{u<v\}} c\theta_v^n. $$
    Taking $b$ large enough so that $(1 - b^{-1})^n > c$, we obtain that  $\theta_{u_b}^n$  is carried by  $D \cap \{u=v\}$.  On this set we have $u_b=v$ and since $u_b \leq bu -(b-1)v \leq v$ on $X$, we obtain again by Corollary \ref{cor: MA inequality}  that 
    $$ \textit{1}_{\{u=v\}\cap D} \theta_{u_b}^n \leq \textit{1}_{\{u=v\}\cap D} \theta_v^n.  $$
  Since the measure $(\theta + dd^cv)^n$ does not charge the pluripolar set $\{v=-\infty\}$, one can construct an increasing function $h : \mathbb{R}_+ \rightarrow \mathbb{R}_+$ such that $h(+\infty)=+\infty$ and $h(|v|) \in L^1(\theta_v^n)$.  That gives
    \begin{align*}
        \int_X h(|u_b|) \theta_{u_b}^n &= \int_{\{u=v\}\cap D} h(|v|) \theta_{u_b}^n \\
        &\leq \int_X h(|v|) \theta_v^n< +\infty.
    \end{align*}
    If $(\sup_X u_b)_b$ has a subsequence that converges to $-\infty$, then for every positive scalar $\alpha$, one can find $b$ such that $\sup_X u_b \leq -\alpha$. That implies 
    $$   \int_X h(|u_b|)\theta_{u_b}^n \geq h(\alpha) v_-(\theta). $$
    This is impossible because 
    $$ \sup_{b\geq 1} \int_X h(|u_b|)\theta_{u_b}^n < \infty,  $$
    and $v_-(\theta)>0$ by Corollary \ref{cor: v- positive}. 
    We conclude that the sequence $(\sup_X u_b)_b$ is uniformly bounded. Therefore, $(u_b)_b$   has a subsequence $(u_{b_j})_j$ that converges in $L^1(X)$ to a function $w \in \PSH(X,\theta)$.

    Fix $a>0$. On $\{u<v-a\}$, we have $u_b \leq v - ab$. Hence   
    $$  \int_{\{u<v-a\}} w \omega^n = \lim_{j\rightarrow +\infty} \int_{\{u<v-a\}} u_{b_j} \omega^n \leq \lim_{j\rightarrow+\infty} -ab_j \int_{\{u<v-a\}} \omega^n + \int_{\{u<v-a\}} v \omega^n. $$
   Since $w \in L^1(X)$, we obtain that  
    $$ \int_{\{u<v-a\}} \omega^n = 0, $$
     and hence $u \geq v - a$  because $u$ and $v$ are quasi-psh. This holds for every positive $a$. Therefore, $u=v$ and the result is thus proved.
\end{proof}
The following corollaries are direct consequences of the domination principle. 
\begin{corollary}\label{cor: uniqness MA-sol}
   Fix $\lambda>0$. If $u_1$, $u_2 \in \mathcal{E}(X,\theta)$ are such that
     $$ e^{-\lambda u_1}(\theta + dd^c u_1)^n  \leq  e^{-\lambda u_2}(\theta + dd^c u_2)^n. $$
     Then $u_1 \geq u_2$.
\end{corollary}
\begin{proof}
    For $a>0$, we have  
 $$
 (\theta + dd^c u_1)^n \leq e^{\lambda (u_1-u_2)}  (\theta + dd^c u_2)^n \leq e^{-\lambda a} (\theta + dd^c u_2)^n, 
 $$
on $\{u_1<u_2-a\}$. Theorem \ref{thm : dom prin} implies $u_1\geq u_2-a$. Since $a$ is arbitrary, letting $a\rightarrow 0$ yields  $u_1\geq u_2$. 
\end{proof}
\begin{corollary}\label{cor: uniq cst c}
    If $u$, $v \in \mathcal{E}(X,\theta)$ are such that 
    $$ (\theta + dd^c u)^n \leq c (\theta + dd^c u)^n, $$
    for some positive constant $c$. Then $c \geq 1$. 
\end{corollary}
\begin{proof}
   Assume in contradiction $c<1$. Then, for every $a \in \mathbb{R}$, we have
   $$ \theta_u^n \leq c(\theta + dd^c(v+a))^n $$
  on X, and hence in particular on $\{u<v+a\}$. Theorem \ref{thm : dom prin} implies $u \geq v+a$. This holds for every $a\in \mathbb{R}$, a contradiction. We infer that $c\geq 1$. 
\end{proof}

Given $u \in \PSH(X,\theta)$, the envelope of singularity types of $u$  is the $\theta$-psh function defined by 
$$ P_{\theta}[u] = \left( \lim_{t \rightarrow +\infty} P_\theta(u+t,0)\right)^*, $$
see \cite{RWN14}. 
The following result provides a second characterization of the full mass class in terms of the envelope of singularity types. 
%The domination principle  that was introduced by Ross and Witt Nystr\"om in . It is defined by 
%$$ P_{\theta}[u] = \left( \lim_{t \rightarrow +\infty} P_\theta(u+t,0)\right)^*, $$
%for every $u \in \PSH(X,\theta)$.  
\begin{theorem}\label{thm: env of sing type}
    A function $u \in \PSH(X,\theta)$ belongs to the full mass class $\mathcal{E}(X,\theta)$ if and only if $P_\theta[u]=0$.
\end{theorem}
\begin{proof}
    Assume $u \in \mathcal{E}(X,\theta)$ and set $\varphi_t = P_\theta(u+t,0)$ for every $t \in \mathbb{R}$.  
    By Theorem \ref{thm: env}, the measure $\theta_{\varphi_t}^n$  is carried by $\{\varphi_t=\min(u+t,0)\}$. Hence, we get by Corollary \ref{cor: MA inequality} that 
    \begin{align*}
        \theta_{\varphi_t}^n &\leq  \textit{1}_{\{\varphi_t =u+t \leq 0\}}  \theta_{\varphi_t}^n +  \textit{1}_{\{\varphi_t =0 \leq u+t\}}  \theta_{\varphi_t}^n \\
        &\leq \textit{1}_{\{u\leq -t\}}  \theta_{u}^n +  \textit{1}_{\{\varphi_t = 0\}}  \theta^n \\
        &\leq \textit{1}_{\{u\leq -t\}}  \theta_{u}^n +  \textit{1}_{\{P_\theta[u] = 0\}}  \theta^n,
    \end{align*}
    because $\varphi_t \leq P_\theta[u] \leq 0$. Observing that $\varphi_t \nearrow P_\theta[u]$ a.e. implies  $\theta_{\varphi_t}^n \rightarrow \theta_{P_\theta[u]}^n$ weakly by Theorem \ref{thm: continuity cap} below. Therefore, letting $t \rightarrow +\infty$ in the last inequality gives
    $$ \theta_{P_\theta[u]}^n \leq \textit{1}_{\{P_\theta[u] = 0\}}  \theta^n, $$
    since $\theta_u^n$ vanishes on pluripolar sets. 
    Hence $\theta_{P_\theta[u]}^n=0$ on $\{P_\theta[u] <0\}$, and consequently $P_\theta[u]=0$ by Theorem \ref{thm : dom prin}.

    Conversely, if $u \in \PSH(X,\theta)$ is such that $P_\theta[u]=0$ then we get $u \in \PSH(X,\omega)$ and $P_\omega[u]=0$ because $\theta \leq \omega$. \cite[Theorem 3.15]{DDL23} thus implies $u \in \mathcal{E}(X,\omega)$, and hence $u \in \mathcal{E}(X,\theta)$ by Proposition \ref{prop: carc full mass}. 
\end{proof}
With Theorem \ref{thm: env of sing type} in hand, one can argue as in the proof of \cite[Theorem 1.1]{DDL1} to show that functions in $\mathcal{E}(X,\theta)$ have zero Lelong numbers at all points.  
\section{Continuity of non-pluripolar Monge-Amp\`ere measures} \label{sect: cont of MA}
%\subsection{Convergence in capacity}
Let $D \subset X$ be a Borel set. The capacity of $D$ with respect to $\omega$ is defined by 
$$ \capa_\omega(D) := \sup \left\{ \int_D (\omega + dd^c h)^n : h \in \PSH(X,\omega), 0\leq h\leq 1 \right\}.  $$
We say that a sequence of functions $f_j:X \rightarrow \mathbb{R}\cup \{\pm\infty\}$ converges in capacity to a function $f$ if 
$$ \lim_{j\rightarrow +\infty} \capa_{\omega}(\{|f_j - f|> \delta \}) = 0, $$
for every $\delta>0$. 
 Notice that if $(f_j)$ is a monotone sequence of quasi-psh functions that converges almost everywhere to a quasi-psh function $f$, then $f_j \rightarrow f$ in capacity by the Hartogs lemma. 
 
 %In constructing weak solutions one often needs to use convergence in capacity of the potentials, which guarantees the convergence of the Monge-Amp\`ere measures. 

%Our first main result in this direction is the following. 

 %By the same idea, we prove a similar result for uniformly bounded sequences. We stress that, in this case, the existence of a Guan-Li metric is not needed. 
 In our result below, we give a sufficient condition for convergence in capacity of sequences of uniformly bounded quasi-psh functions, positively answering a question of Ko{\l}odziej-Nguyen \cite{KN22}.  
\begin{theorem}\label{thm: cv cap bounded}
 Let $u_j \in \PSH(X,\theta)$ be a uniformly bounded sequence converging in $L^1(X)$ to a $\theta$-psh function $u$. If 
 $$ \lim_{j\rightarrow +\infty} \int_X |u_j - u| \theta_{u_j}^n = 0, $$
 then $u_j \rightarrow u$ in capacity. 
\end{theorem}
\begin{proof}
    Up to extracting a subsequence, one can assume 
    $$ \int_X |u_j - u| \theta_{u_j}^n \leq 2^{-j}, \; \; \forall j. $$
    Fix $j \in \mathbb{N}$ and consider $v_{j,k}= P_\theta(\min(u_j,..,u_k))$ for every  $k \geq j$. The sequence $(v_{j,k})_k$ decreases to the function $v_j = P_\theta(\inf_{k\geq j} u_k)$. It follows from Lemma \ref{lem: MA of rooftop env 2} that
     $$ \int_X |v_{j,k} - u| \theta_{v_{j,k}}^n \leq \sum_{l=j}^k \int_X |u_l - u| \theta_{u_l}^n.  $$
%  
   % For $D_l=\{v_{j,k}=u_l<\min_{j\leq m\leq l-1} u_m\}$, we have $D_l \cap D_{l'} = \emptyset$ whenever $l \neq l'$ and the measure $\theta_{v_{j,k}}^n$ is carried by $\cup_l D_l$ according to Theorem \ref{thm: env}. Notice that, for every $l$,  $v_{j,k} \leq u_l$ with equality on $D_l$. Hence, we get by Corollary \ref{cor: MA inequality} that
    %$$ \int_X |v_{j,k} - u| \theta_{v_{j,k}}^n \leq \sum_{l=j}^k \int_{D_l} |u_l - u| \theta_{u_l}^n \leq 2^{-j+1}.  $$
    Therefore, using the Bedford-Taylor convergence theorem for decreasing and  uniformly bounded sequences along with \cite[Lemma 2.5]{DDL23}, we obtain that
    $$ \int_X |v_j - u| \theta_{v_j}^n = \lim_{k\rightarrow +\infty} \int_X |v_{j,k} - u| \theta_{v_{j,k}}^n \leq 2^{-j+1}. $$
    The sequence $(v_j)$ is non-decreasing. Let $v \in \PSH(X,\theta)$ be such that $v_j \rightarrow v$ almost everywhere. We have again by the Bedford-Taylor convergence theorem and \cite[Lemma 2.5]{DDL23} that 
    $$ \int_X |v-u| \theta_v^n = \lim_{j\rightarrow +\infty} \int_X |v_j - u| \theta_{v_j}^n = 0,$$
   which implies $v \geq u$ by \cite[Proposition 2.8]{GL22}. Since $v_j \leq u_j$ for all $j$, we infer that $u=v$. Since $v_j\nearrow u$ and $v_j\leq u_j$, it follows that $u_j \rightarrow u$ in capacity, completing the proof. 
\end{proof}

We next prove that the non-pluripolar Monge-Amp\`ere measure is continuous with respect to convergence in capacity. 

 \begin{theorem}\label{thm: continuity cap}
    Assume $u_j$, $u \in \mathcal{E}(X,\theta)$ are such that  $u_j \rightarrow u$ in capacity. Assume $h_j$ is a sequence of uniformly bounded quasi-continuous functions converging in capacity to a bounded quasi-continuous function $h$. Then  $h_j\theta_{u_j}^n$ weakly converges to $h\theta_u^n$.
\end{theorem}
\begin{proof}
Let $\Theta$ be a  cluster point of the sequence $\theta_{u_j}^n$ for the weak topology. We prove that $\Theta = \theta_u^n$. Up to extracting a subsequence, one can assume  $\theta_{u_j}^n \rightarrow \Theta$ weakly. 
  According to \cite[Theorem 2.6]{DDL23}, it suffices to prove 
  $$ \int_X \theta_u^n \geq \limsup_{j\rightarrow +\infty} \int_X \theta_{u_j}^n = \Theta(X).  $$
  Fix  $j \in \mathbb{N}$ and $k \in \{0,...,n\}$. Setting  
  $$ m^k = \int_X \theta_u^k \wedge (\omega - \theta)^{n-k} \; \text{and} \; m_j^k = \int_X \theta_{u_j}^k \wedge (\omega - \theta)^{n-k}, $$
   we obtain $m^k \leq \liminf_{j\rightarrow +\infty} m_j^k$   by \cite[Theorem 2.6]{DDL23}.
%  $$  \int_X \theta_u^k \wedge (\omega - \theta)^{n-k} \leq \liminf_{j\rightarrow +\infty} \int_X \theta_{u_j}^k\wedge (\omega - \theta)^{n-k}, $$
 % for every $k=0,..,n$. 
  
 % For simplicity, we set
 % $$ m^k = \int_X \theta_u^k \wedge (\omega - \theta)^{n-k} \; \text{and} \; m_j^k = \int_X \theta_{u_j}^k \wedge (\omega - \theta)^{n-k}. $$
  On the other hand, we get by Proposition \ref{prop: carc full mass} that $u_j$, $u \in \mathcal{E}(X,\omega)$, and consequently 
  $$ \int_X \omega_{u_j}^n = \int_X \omega_{u}^n =  \int_X \omega^n.  $$
  Using the Newton expansion formula, we obtain
  $$ \sum_{k=0}^n \binom{n}{k} m^k = \sum_{k=0}^n \binom{n}{k} m_j^k. $$
 % and 
  %$$ \sum_{k=0}^n \binom{n}{k} m_j^k =  \int_X \omega_{u_j}^n  = \int_X \omega^n.  $$
 % for every $j$. 
 %  Multiplying each term by $\binom{n}{k}$ and summing from $k=0$ to $k=n$ yields  
 %$$    \int_X \omega^n =  \sum_{k=0}^n \binom{n}{k} \int_X \theta_u^k \wedge (\omega - \theta)^{n-k}  \leq \liminf_{j\rightarrow +\infty} \sum_{k=0}^n \binom{n}{k} \int_X \theta_{u_j}^k \wedge (\omega - \theta)^{n-k}.$$
    Therefore, we infer that $m^k = \liminf_{j\rightarrow +\infty} m_j^k$ for %every $k$. 
  %$$ \int_X \theta_u^k \wedge (\omega - \theta)^{n-k} = \liminf_{j\rightarrow +\infty} \int_X \theta_{u_j}^k\wedge (\omega - \theta)^{n-k}, $$
   every $k$. In particular for $k=n$, we have
 $$   \int_X \theta_u^n = \liminf_{j\rightarrow +\infty}     \int_X \theta_{u_j}^n = \Theta(X).$$
The result thus follows by applying \cite[Theorem 2.6]{DDL23}.
%yielding the weak convergence $\theta_{u_j}^n \to \theta_u^n$. By the proof of \cite[Theorem 2.6]{DDL23}, we actually have  $h_j\theta_{u_j}^n \to h\theta_u^n$. 
%
%We now treat the general case. Let $C$ be a constant such that $\sup(|h|,|h_j|)\leq C$. Then, by \cite[Theorem 2.6]{DDL23},
%\[
%\liminf (h_j+C)\theta_{u_j}^n \geq (h+C)\theta_u^n\; \; \text{and}\; \liminf (C-h_j)\theta_{u_j}^n \geq (C-h)\theta_u^n. 
%\]
%Summing up we obtain 
%\[
%\liminf \theta_{u_j}^n \geq  \theta_u^n,
%\]
%Since $\lim_j \theta_{u_j}^n= \theta_u^n$, all the above inequalities become equalities, yielding the result. 
\end{proof}

The following result gives a sufficient condition to ensure the convergence in capacity of $\theta$-psh functions in $\mathcal{E}(X,\theta)$, generalizing Theorem \ref{thm: cv cap bounded}. 
\begin{theorem}\label{thm cvg cap unbounded}
    Assume $u_j \in \mathcal{E}(X,\theta)$ is such  that   $\theta_{u_j}^n \leq \mu$ for some non-pluripolar Radon measure $\mu$. If $u_j \rightarrow u \in \PSH(X,\theta)$ in $L^1(X)$, then $u \in \mathcal{E}(X,\theta)$ and $u_j \rightarrow u$ in capacity. 
\end{theorem}

We will need the following lemmas. 

\begin{lemma}\label{lem: prep 1}
    Assume $u_j\in  \mathcal{E}(X,\theta)$ is such that $\theta_{u_j}^n \leq \mu$ for some Radon measure vanishing on pluripolar sets. Then, for any $v\in \PSH(X,\theta)$, we have
    \[
    \inf_{j} \int_X e^v \theta_{u_j}^n >0. 
    \]
\end{lemma}
\begin{proof}
    Assume in the contrary that 
    $$  \inf_{j} \int_X e^v \theta_{u_j}^n =0.  $$
    Up to extracting a subsequence, one can assume 
    $$ \lim_{j\rightarrow+\infty}  \int_X e^v \theta_{u_j}^n = 0.  $$
   Fixing $C>0$, we have
   $$ v_-(\theta) \leq \int_X \theta_{u_j}^n = \int_{\{v \leq -C\}} \theta_{u_j}^n + \int_{\{v > -C\}} \theta_{u_j}^n \leq  \mu(\{v \leq -C\} + e^C \int_X e^v \theta_{u_j}^n. $$
   Letting $j \rightarrow +\infty$ and then $C\rightarrow +\infty$, we obtain $v_-(\theta)=0$. This contradicts Corollary \ref{cor: v- positive}. The result follows. 
\end{proof}

\begin{lemma}\label{lem: prep 2}
    Assume $u_j\in  \mathcal{E}(X,\theta)$  satisfy the hypothesis of the last theorem and $u_j\to u\in \PSH(X,\theta)$ in $L^1(X)$.  Then $u \in \mathcal{E}(X,\theta)$. 
\end{lemma}
\begin{proof}
 %   Assume by contradiction that $u_j \rightarrow - \infty$.  By our assumptions, this implies 
  %  $$ \lim_{j\rightarrow +\infty} \int_X  e^v \theta_{u_j}^n = 0. $$
  %   Fixing $C>0$, we get 
  %  $$ v_-(\theta) \leq \int_X \theta_{u_j}^n \leq \int_{\{ v \leq -C\}} \theta_{u_j}^n +  \int_{\{v > -C\}} \theta_{u_j}^n \leq  \int_{\{v \leq -C\}} d\mu  +  e^{C} \int_X e^{v} \theta_{u_j}^n. $$
  % Letting $j \rightarrow +\infty$ and then $C\rightarrow +\infty$ gives $v_-(\theta)=0$, contrary to Corollary \ref{cor: v- positive}. Therefore $(u_j)$ has a subsequence that converges to a function $u \in \PSH(X,\theta)$ in $L^1(X)$. 
  % Up to extracting a subsequence, we can assume $u_j \rightarrow u$ in $L^1(X)$ and almost everywhere.  
    
    According to Proposition \ref{prop: carc full mass}, it suffices to prove $P_\theta(Au) \in \PSH(X,\theta)$ for every $A\geq 1$.
    Fix $A\geq 1$ and set $\varphi_j = P_\theta(Au_j)$. It follows from Theorem \ref{thm: env} that $\theta_{\varphi_j}^n$ is carried by $\{\varphi_j = A u_j\}$ and that $\varphi_j \leq Au_j$.  
     Corollary \ref{cor: MA inequality} implies 
    $$ \theta_{\varphi_j}^n = \textit{1}_{\{\varphi_j = A u_j\}} \theta_{\varphi_j}^n \leq \textit{1}_{\{\varphi_j = A u_j\}} (A\theta+dd^c{\varphi_j})^n \leq  A^n \theta_{u_j}^n \leq A^n \mu. $$
    Hence, we obtain
 %  $$       \int_X |e^{A^{-1}\varphi_j} - e^{v}| \theta_{\varphi_j}^n       \leq A^n \int_{\{\varphi_j=Au_j\}} |e^{u_j} - e^{v}| \theta_{u_j}^n.   $$
    %Hence, we obtain
    $$ \lim_{j\rightarrow +\infty} \int_X |e^{A^{-1}\varphi_j} - e^{u}| \theta_{\varphi_j}^n \leq  A^n \lim_{j\rightarrow +\infty} \int_X |e^{u_j} - e^u| d\mu = 0, $$
    by \cite[Lemma 11.5]{GZ17} and \cite[Corollary 2.2]{KN22}. Assume by contradiction that $\varphi_j$ converges uniformly to $-\infty$. That implies 
    $$  \lim_{j\rightarrow +\infty} \int_X e^{A^{-1}\varphi_j} \theta_{\varphi_j}^n \leq \lim_{j\rightarrow +\infty} \int_X e^{A^{-1}\varphi_j} d\mu = 0,  $$
    by the Lebesgue convergence theorem.  Therefore, we obtain 
    $$  \lim_{j\rightarrow +\infty} \int_X  e^{u} \theta_{\varphi_j}^n = 0,$$
    contradicting Lemma \ref{lem: prep 1}. 
  %   Fixing $C>0$, we get 
  %  $$ v_-(\theta) \leq \int_X \theta_{\varphi_j}^n \leq \int_{\{ u \leq -C\}} \theta_{\varphi_j}^n +  \int_{\{u > -C\}} \theta_{\varphi_j}^n \leq A^n \int_{\{v \leq -C\}} d\mu  +  e^{C} \int_X e^{v} \theta_{\varphi_j}^n. $$
%   Letting $j \rightarrow +\infty$ and then $C\rightarrow +\infty$ gives $v_-(\theta)=0$, contrary to Corollary \ref{cor: v- positive}. 
It thus follows that $(\varphi_j)$ has a subsequence converging to a function $\varphi \in \PSH(X,\theta)$.  Since $\varphi_j \leq A u_j$ for every $j$, we conclude that $\varphi \leq A u$ and hence  $P_\theta(Au) \in \PSH(X,\theta)$. The result therefore follows by Proposition \ref{prop: carc full mass}.  
\end{proof}
\begin{lemma}\label{lem: prep 3}
    Consider $u_j$, $u$ as in Theorem \ref{thm cvg cap unbounded}. Then $(u_j)$ has a subsequence converging to $u$ in capacity. 
\end{lemma}
\begin{proof}
     We can assume $u_j \leq 0$ for every $j$. The last lemma implies that  $u \in \mathcal{E}(X,\theta)$, and we have by the hypothesis that $u_j \rightarrow u$ in $L^1(X)$. We thus have $e^{u_j}\to e^u$ in $L^1$ and these are uniformly bounded $\theta$-psh functions.  It thus follows from \cite[Lemma 11.5]{GZ17} and \cite[Corollary 2.2]{KN22} that
    $$ \lim_{j\rightarrow +\infty} \int_X |e^{u_j} - e^u| d\mu = 0.  $$
    Hence, up to extracting a subsequence, one can assume 
    $$ \int_X |e^{u_j} - e^u| d\mu \leq 2^{-j}, \; \; \forall j\geq 1. $$ 
    For each $k \geq j \geq 1$, consider $v_{j,k}=P_\theta(\min(u_j,..,u_k))$. 
%   By Proposition \ref{prop : stab maxima}, the function $h=(u_j+...+u_k)/k$ belongs to $\mathcal{E}(X,\theta)$. Since $v_{j,k} \geq  P_{\theta}(kh)$ and $P_{\theta}(kh) \in \mathcal{E}(X,\omega)$ by Corollary \ref{cor: P(Au) in full mass class}, we get 
It follows from Lemma \ref{lem: MA of rooftop env 1} and Lemma \ref{lem: MA of rooftop env 2} that $v_{j,k} \in \mathcal{E}(X,\theta)$,  $\theta_{v_{j,k}}^n \leq \mu$, and 
\[
\int_X |e^{v_{j,k}}-e^u| \theta_{v_{j,k}}^n \leq 2^{-j+1}. 
\]
By Lemma \ref{lem: prep 1}, $\inf_{j,k}\int_X e^u \theta_{v_{j,k}}^n>0$, hence $\int_X e^{v_{j,k}}\theta_{v_{j,k}}^n$ does not converge to $0$ as $k\to +\infty$. Since $\int_X\theta_{v_{j,k}}^n\leq v_+(\theta)$, we infer that $\sup_X v_{j,k}$ does not converge to $-\infty$ as $k\to +\infty$. 
    Note also that $(v_{j,k})_k$ is a decreasing sequence, hence  $v_j := P_\theta(\inf_{k\geq j} u_k) \in \PSH(X,\theta)$.  Lemma \ref{lem: prep 2} then gives $v_j\in \mathcal{E}(X,\theta)$. 
   Hence, we get by the monotone convergence theorem that  
    $$ \lim_{k \rightarrow +\infty} \int_X |e^{v_{j,k}} - e^{v_j}|\theta_{v_{j,k}}^n \leq \lim_{k \rightarrow +\infty} \int_X |e^{v_{j,k}} - e^{v_j}|d\mu = 0.  $$   
 
    %  According to the last lemma we have $v_j \in \mathcal{E}(X,\theta)$ or $v_j=-\infty$. 
    %  Assume by contradiction that  $v_j=-\infty$.  The previous identity yields 
   % $$ \lim_{k\rightarrow+\infty} \int_X e^{v_{j,k}} \theta_{v_{j,k}}^n = 0. $$
   % On the other hand, we have by Lemma \ref{lem: MA of rooftop env 2} that 
   % $$ \int_X |e^{v_{j,k}} - e^u| \theta_{v_{j,k}}^n \leq \sum_{l=j}^k \int_{X} |e^{u_l} - e^u| d\mu \leq 2^{-j+1}.  $$
 %   Hence, for $k\rightarrow +\infty$, we obtain 
 %   $$ \lim_{k\rightarrow +\infty} \int_X  e^u \theta_{v_{j,k}}^n = \lim_{k\rightarrow +\infty} \int_X |e^{v_{j,k}} - e^u|\theta_{v_{j,k}}^n \leq 2^{-j+1}. $$
%Hence, 
%$$ \inf_{j,k} \int_X  e^u \theta_{v_{j,k}}^n = 0. $$
%This is impossible because of Lemma \ref{lem: prep 1}.
%Thus, we infer that 

%    Fixing $C>0$, we have 
%   $$ v_-(\theta) \leq \int_X \theta_{v_{j,k}}^n \leq \mu(\{u\leq - C\}) + e^C\int_X e^u \theta_{v_{j,k}}^n.  $$
%    Letting $k \rightarrow +\infty$ yields
%    $$  v_-(\theta) \leq \mu(\{u\leq - C\}) + 2^{-j+1}e^C. $$
%    Letting $j \rightarrow +\infty$ and then $C \rightarrow +\infty$ implies $v_-(\theta)=0$.  This contradicts Corollary \ref{cor: v- positive}.
%Therefore, we infer that $v_j \in \mathcal{E}(X,\theta)$. 

    By construction, the sequence $(v_j)$ is non-decreasing; set $v:= \lim_jv_j$. We then  have $v\in \mathcal{E}(X,\theta)$ by Proposition \ref{prop : stab maxima}. It thus follows from \cite[By Proposition 4.25]{GZ07} that  $v_j$ converges to $v$ in capacity.
    
    Now, observe that the functions $|e^{v_j} - e^u|$ are uniformly bounded, quasi-continuous, and converge in capacity to $|e^v - e^u|$. Therefore, we get by Theorem \ref{thm: continuity cap} that 
    $$ \int_X |e^v-e^u| \theta_v^n = \lim_{j\rightarrow +\infty} \int_X |e^{v_j}-e^u| \theta_{v_j}^n. $$ 
    Similarly, we have that 
     $$ \int_X |e^{v_j}-e^u| \theta_{v_j}^n = \lim_{j\rightarrow +\infty} \int_X |e^{v_{j,k}}-e^u| \theta_{v_{j,k}}^n =0,$$ 
    hence
    $$ \int_X |e^u - e^v| \theta_v^n = 0. $$ 
    The domination principle then gives $v\geq u$. By construction of $v$, we have $u_j\leq v_j$, hence $u\leq v$, and thus $u=v$.

    Finally, since $v_j \leq u_j\leq \max(u_j,u)$ for every $j$, and  $(v_j)$, $(\max(u_j,u))$ converge in capacity to $u$, we infer that $u_j \rightarrow u$ in capacity.      
    \end{proof} 
Now we proceed to the proof of Theorem \ref{thm cvg cap unbounded}. 
    \begin{proof}[Proof of Theorem \ref{thm cvg cap unbounded}]
    Assume in the contrary that $(u_j)$ does not converge in capacity to $u$. That implies the existence of positive constants $\varepsilon$ and  $\delta$, and a subsequence $(v_j) \subset (u_j)$  such that
 $$       \capa_\omega(\{|v_j - u|>\delta\}) \geq \varepsilon. $$
   We obtain a contradiction by applying the last lemma to the sequence $(v_j)$. Hence $u_j \rightarrow u$ in capacity, and the proof is complete. 
\end{proof}

\section{Weak solutions to Monge-Ampère equations}\label{sect: solutions to MA-E}

The main goal of this section is to study the existence of weak solutions to the equation 
\begin{equation}\label{MA-eqt}
     \theta_u^n = ce^{\lambda u}\mu, \; (u,c) \in \mathcal{E}(X,\theta) \times \mathbb{R}^*_+.   
     \end{equation}

Here $\lambda \in \mathbb{R}^+$ is a fixed constant, and $\mu$ is a positive Radon measure vanishing on pluripolar sets.

First, we start by proving that the measure $\mu$ is absolutely continuous with respect to the Monge-Ampère measure of a bounded $\omega$-psh function. 
\begin{theorem}\label{thm 3.1}
    There is $\phi \in \PSH(X,\omega) \cap L^\infty(X)$ and $f \in L^1((\omega+dd^c \phi)^n)$ such that  $$ \mu = f (\omega + dd^c \phi)^n. $$
\end{theorem}
\begin{proof}
Fix open sets $U\Subset V\subset X$ such that $V$ is biholomorphic to the unit ball in $\mathbb{C}^n$ and $U$ is biholomorphic to the ball of radius $1/2$. By \cite[Theorem 5.11]{Ceg04}, there exists $u\in \PSH(V)$, $-1\leq u\leq 0$, such that $\mu$ is absolutely continuous with respect to $(dd^c u)^n$.  Let $\rho$ be a smooth function such that $\rho=-1$ in $U$ and $\rho=0$ in a neighborhood of $X\setminus V$. The function $v:=\max(u,\rho)$ is $A\omega$-psh in $U$ which coincides with $u$ in $U$ and $\rho$ near $\partial V$. It is thus $A\omega$-psh on $X$, hence $A^{-1}v$ is $\omega$-psh on $X$ and on $U$, $\mu$ is absolutely continuous with respect to $(\omega+dd^c v)^n$.  

Consider now a double finite covering of $X$ by  $U_j \Subset V_j$, $j=1,...,N$. Arguing as above, we obtain bounded $\omega$-psh functions $v_1,...,v_N$ such that on $U_j$ $\mu$ is absolutely continuous with respect to $(\omega+dd^c v_j)^n$, hence also with respect to $(\omega+dd^c v)^n$, where 
\[
v:= N^{-1}(v_1+...+v_N). 
\]
Since the open sets $U_j$ cover $X$, the result follows. 

\end{proof}

\begin{theorem}\label{thm 3.2}
    Fix $\lambda>0$ and let $\mu$ be a positive Radon measure vanishing on pluripolar sets. Then there is a unique $u \in \mathcal{E}(X,\theta)$ such that 
    $$ (\theta + dd^c u)^n = e^{\lambda u}\mu. $$
\end{theorem}
%The proof of the last theorem relies on the following lemma. 
%\begin{prop}\label{prop 3.6} 
 %   Let $\mu_j$, $\mu$ be positive Radon measures vanishing on pluripolar sets. If $\mu_j \rightarrow \mu$ weakly, then 
 %   $$ \liminf_{j \rightarrow +\infty} \int_X e^{\lambda u} d\mu_j >0, $$
 %   for every $u \in \PSH(X,\omega)$ and every $\lambda \geq 0$.
%\end{prop}
%\begin{proof}
 %   Assume in the contrary that 
 %   $$ \lim_{j\rightarrow +\infty}  \int_X e^{\lambda u} d\mu_j = 0. $$
 %   Fix $k \in \mathbb{N}$. We have by \cite[Lemma 2.4]{DDL23} that %   \begin{align*}
 %       \mu(\{u > -k\}) &\leq \liminf_{j\rightarrow +\infty} \mu_j(\{u > -k\}) \\
 %       &\leq e^{\lambda k} \liminf_{j\rightarrow +\infty} \int_X e^{\lambda u} d\mu_j \\
 %       &= 0.  
 %   \end{align*}
 %   Hence 
 %   $$ \mu(X) = \lim_{k \rightarrow +\infty} \mu(\{u>-k\}) = 0, $$
 %   a contradiction. 
%\end{proof}
%Now we move on to the proof of Theorem \ref{thm 3.2}
\begin{proof}
By the last theorem, one can write $\mu = f(\omega + dd^c \phi)^n$, where $\phi \in \PSH(X,\omega)\cap L^\infty(X)$ and $f\in L^1((\omega + dd^c \phi)^n)$.
The proof will be devided in two steps.

    {\bf Step 1.} We first assume $f \in L^\infty(X)$. Let $\phi_j \in \PSH(X,\omega)\cap \mathcal{C}^\infty(X)$, $\phi_j \searrow \phi$. According to \cite[Theorem 3.4]{GL23Crelle}, for every $j\geq 1$, there is $u_j \in \PSH(X,\theta)\cap L^\infty(X)$ such that 
    $$ \theta_{u_j}^n = e^{\lambda u_j} f (\omega + dd^c \phi_j)^n. $$
  Since $f$, $(\phi_j)$ are uniformly bounded, we get 
  by  Corollary \ref{cor: uniqness MA-sol} that there is $C>0$ such that $\phi_j - C \leq u_j \leq \phi_j + C$. Hence, the sequence $(u_j)$ is uniformly bounded.  Up to extracting a subsequence, one can assume $u_j \rightarrow u \in \PSH(X,\theta)\cap L^\infty$ in $L^1(X)$ and almost everywhere. Since $f$ is bounded, we get 
   $$ (\theta + dd^c u_j)^n \leq A (\omega + dd^c \phi_j)^n, $$
   for some positive constant $A$. Hence, we obtain by \cite[Lemma 2.3]{KN22} that 
   $$ \lim_{j\rightarrow +\infty} \int_X |u_j - u| (\theta + dd^c u_j)^n = 0.  $$
   Theorem \ref{thm: cv cap bounded} implies that $u_j \rightarrow u$ in capacity, and hence $\theta_{u_j}^n \rightarrow \theta_u^n$ weakly by Theorem \ref{thm: continuity cap}. Moreover, we have  
  $$ e^{\lambda u_j} f \omega_{\phi_j}^n \rightarrow e^{\lambda u} f \omega_\phi^n, $$
  weakly according to \cite[Lemma 2.5]{DDL23}, which implies $\theta_u^n = e^{\lambda u} f \omega_\phi^n$.

%    {\bf Step 2.} We assume $f \in L^\infty(X)$. Let $f_j \in \mathcal{C}^{\infty}(X)$, $f_j \rightarrow f$ in $L^1(\omega_\phi^n)$. By Step 1, there is $u_j \in \PSH(X,\omega)\cap L^\infty(X)$ such that 
%    $$ (\omega + dd^c u_j)^n = e^{\lambda u_j} f_j (\omega + dd^c \phi)^n. $$
%   Since $\phi$, $f_j$ are uniformly bounded, we infer from the domination principle   that the sequence $(u_j)$ is uniformly bounded on $X$ and that  $\omega_{u_j}^n \leq C \omega_\phi^n$. 
%   Up to extracting a subsequence, one can assume $u_j \rightarrow u$ in $L^1(X)$ and almost everywhere. Theorem \ref{thm: cv cap bounded} implies $u_j \rightarrow u$ in capacity. Therefore $\omega_{u_j}^n \rightarrow \omega_{u}^n$ according to Theorem \ref{thm: continuity cap}. We get by the Lebesgue convergence theorem that 
%   $$ \omega_{u_j}^n = e^{\lambda u_j} f_j \omega_{\phi_j}^n \rightarrow e^{\lambda u} \mu, $$
%   strongly. Therefore, we obtain $\omega_u^n = e^{\lambda u}\mu$.  

    {\bf Step 2.} We move on the general case when $\phi \in \PSH(X,\omega) \cap L^\infty$, $f \in L^1(\omega_\phi^n)$.
  By Step 1,  for each $j \in \mathbb{N}^*$,  there is $u_j \in \PSH(X,\theta) \cap L^\infty$ such that 
    $$ (\theta + dd^c u_j)^n = e^{\lambda u_j} \min(f,j) (\omega + dd^c \phi)^n. $$
    The sequence $(u_j)$ is decreasing by Corollary \ref{cor: uniqness MA-sol}; set $u = \lim u_j$. We prove that $u \in \mathcal{E}(X,\theta)$. 
    Notice that, if $sup_X u_j\to -\infty$ then 
    \[
   v_-(\theta) \leq  \int_X \theta_{u_j}^n \leq \int_X e^{\lambda \sup_X u_j} d\mu \to 0,
    \]
    contradicting Corollary \ref{cor: v- positive}. 
%    Assume by contradiction that  $u\equiv -\infty$. Fixing $\alpha>0$, there is $j_0$ such that $v_j \leq -\alpha$ for every $j \geq j_0$. 
 %   Notice that, we have  
  %   $$ \int_X -u_j \theta_{u_j}^n \leq \int_X -u_j e^{\lambda u_j} d\mu \leq \mu(X)/\lambda, \; \; \forall j \geq j_0. $$
 %      That implies 
 %  $$ \frac{1}{\lambda}\mu(X) \geq \int_X -u_j \theta_{u_j}^n \geq \alpha \int_X \theta_{u_j}^n \geq \alpha v_-(\theta).  $$
  % This is impossible because $v_-(\theta)>0$ by Proposition \ref{cor: v- positive}.
  We conclude that $u \in \PSH(X,\theta)$.  Since $\theta_{u_j}^n \leq e^{\lambda u_1} d\mu$ for every $j\geq 1$, we get by Theorem \ref{thm cvg cap unbounded} that $u \in \mathcal{E}(X,\theta)$, and therefore  $\theta_{u_j}^n \rightarrow \theta_u^n$ weakly according to Theorem \ref{thm: continuity cap}. Furthermore, we obtain by the Lebesgue convergence theorem that 
   $$ \theta_{u_j}^n =  e^{\lambda u_j} \min(f,j) (\omega + dd^c \phi)^n \rightarrow e^{\lambda u} d\mu,   $$
   in the strong sense of measures. That implies $\theta_u^n = e^{\lambda u} d\mu$. The uniqueness of the solution follows from Corollary \ref{cor: uniqness MA-sol}.

\end{proof}

\begin{corollary}
 Let $\mu$ be a positive Radon vanishing on pluripolar sets. Then, there is a unique $c>0$ and a function $u \in \mathcal{E}(X,\theta)$ such that 
    $$ \theta_u^n = c d\mu. $$    
\end{corollary}
\begin{proof}
    By the last theorem, for every $j\geq 1$, there is  $u_j \in \mathcal{E}(X,\theta)$ such that
    $$ \theta_{u_j}^n = e^{u_j/j} d\mu. $$
    Setting $v_j= u_j - \sup_X u_j$, one can write 
    $$ \theta_{v_j}^n = c_j e^{v_j/j} d\mu, $$
    for some positive constant $c_j$.
  Up to extracting a subsequence one can assume $v_j \rightarrow v \in \PSH(X,\theta)$ in $L^1(X)$ and almost everywhere. 
  It thus follows from \cite[Lemma 11.5]{GZ17} and \cite[Corollary 2.2]{KN22} that $e^{v_j/j} d\mu \rightarrow \mu$ strongly. 
   Since for every $j$,
  $$ 0 < v_-(\theta) \leq   \int_X \theta_{v_j}^n = c_j \int_X e^{v_j/j} d\mu \leq v_+(\theta)< +\infty, $$
    we infer that $(c_j)$ is uniformly bounded. Therefore, up to extracting a subsequence, one can suppose $c_j \rightarrow c>0$. Hence, we get
  $$ \theta_{v_j}^n = c_j e^{v_j/j} \mu \rightarrow c\mu, $$
   as $j\rightarrow +\infty$ strongly. On the other hand, 
    we have $\theta_{v_j}^n \leq c_j \mu \leq C\mu$ for some positive constant $C$. 
   Therefore, it follows from Theorem \ref{thm cvg cap unbounded} that  $v \in \mathcal{E}(X,\theta)$ and $v_j \rightarrow v$ in capacity. Hence we infer $\theta_{v_j}^n \rightarrow \theta_v^n$ weakly by Theorem \ref{thm: continuity cap},
   which implies $\theta_v^n = cd\mu$. The uniqueness of the constant $c$ follows from Corollary \ref{cor: uniq cst c}. 
\end{proof}

%\bibliographystyle{amsalpha_nodash}
%\bibliography{Biblio}
\newcommand{\etalchar}[1]{$^{#1}$}
\providecommand{\bysame}{\leavevmode\hbox to3em{\hrulefill}\thinspace}
\providecommand{\MR}{\relax\ifhmode\unskip\space\fi MR }
\providecommand{\MRhref}[2]{%
  \href{http://www.ams.org/mathscinet-getitem?mr=#1}{#2}
}
\providecommand{\href}[2]{#2}

\end{document}